\documentclass[12pt]{article}
\usepackage{amssymb,amsmath}
\usepackage{latexsym,bm}

\usepackage{amsfonts}
\usepackage{latexsym,amsmath,amssymb,amsfonts,epsfig,psfrag,url,graphics,ifpdf,multicol}
\usepackage{eepic,color,colordvi,amscd,amsthm}
\usepackage[section]{algorithm}
\usepackage{algpseudocode}
\usepackage[numbers,sort&compress]{natbib}
\usepackage{indentfirst,graphics,epsfig}
\usepackage{graphicx}
\usepackage{graphics}
\usepackage{tikz}

\newtheorem{theo}{Theorem}
\newtheorem{lemma}[theo]{Lemma}
\newtheorem{corollary}[theo]{Corollary}

\newtheorem{problem}[theo]{Problem}
 \theoremstyle{definition}

\theoremstyle{remark}

\def\g{\overrightarrow{g}}
\def\w{\overrightarrow{w}}

\def\C{\overrightarrow{C}}
\def\V{\mathcal{V}}
\def\A{\mathcal{A}}
\newcounter{casenum}[theo]

\newcounter{subcasenum}[theo]

\newcounter{claimnum}[theo]

\allowdisplaybreaks
\usepackage{indentfirst,subfig}

\begin{document}
\thispagestyle{plain}

\begin{center} {\Large The  set of stable indices of 0-1 matrices with a given order
}
\end{center}
\pagestyle{plain}
\begin{center}
{
  {\small Zhibing Chen, Zejun Huang\footnote{Corresponding author ( Email:   mathzejun@gmail.com)     }}\\
  {\small College of Mathematics and Statistics, Shenzhen University, Shenzhen 518060, China }\\

}
\end{center}

\begin{center}

\begin{minipage}{140mm}
\begin{center}
{\bf Abstract}
\end{center}
{\small   The stable index of a 0-1 matrix  $A$ is defined to be the smallest integer $k$ such that $A^{k+1}$  is not a 0-1 matrix if such an integer exists; otherwise the stable index of $A$ is defined to be infinity.   We characterize the set of stable indices of 0-1 matrices  with a given order.

 }

\end{minipage}
\end{center}

\section{Introduction and Main Results }

    Combinatorial problems on the power of nonnegative matrices is an interesting topic in matrix theory. One of the  classical problems in this topic is the characterization of the exponents of primitive matrices with a given order. Let $A$ be a nonnegative matrix. Denote by $\rho(A)$ the spectral radius of $A$. By the famous Perron-Frobenius theorem we know that $\rho(A)$ is an eigenvalue of $A$. If $A$ has no other eigenvalue of modulus $\rho(A)$, then   $A$  is called {\it primitive}. Frobenius  proved that a nonnegative square matrix is primitive if and only if there is a positive integer $k$ such that $A^k$ is positive entrywise. The {\it exponent} of a primitive matrix is  the smallest positive integer $k$ such that $A^k$ is positive. A natural interesting problem is which numbers are the exponents of nonnegative matrices with a given order, which has been solved by  Wielandt \cite{HW}, Dulmage and Mendelsohn \cite{DM}, Lewin and  Vitek \cite{LV}, Shao \cite{JS1} and Zhang \cite{KZ}. It turns out that the set of exponents of primitive matrices of order $n$ is not a consecutive set in the interval $[1, (n-1)^2+1]$, which is quite surprising.

  Analogous to the exponent of primitive matrices, Chen, Huang and Yan \cite{CHY} introduced the concept of {\it stable index} for 0-1 matrices as follows.

 {\bf Definition.} {\it  The stable index of a 0-1 matrix  $A$, denoted by $\theta(A)$, is defined to be the smallest integer $k$ such that $A^{k+1}$  is not a 0-1 matrix if such an integer exists; otherwise the stable index of $A$ is defined to be infinity.
   }

 The stable index of 0-1 matrices is closely related to directed walks of digraphs.   Digraphs in this paper allow loops but do not allow multiple arcs. We follow the terminology on digraphs in \cite{BM}. Directed paths, directed cycles
and directed walks will be abbreviated as paths, cycles and walks, respectively. A path (walk) with initial vertex $u$ and terminal vertex $v$ is called a $uv$-path ($uv$-walk).   A walk (path, cycle) of length $k$ is called a {\it $k$-walk ($k$-path, $k$-cycle)}. The number of vertices in a digraph is called its {\it order}.

  Denote by $M_n\{0,1\}$ the set of 0-1 matrices of order $n$.
Given a matrix $A=(a_{ij})\in M_n\{0,1\}$, we   define its digraph as $D(A)=(\mathcal{V},\mathcal{A})$  with the vertex set $\mathcal{V}=\{1,2,\ldots,n\}$ and the arc set $\mathcal{A}=\{(i,j): a_{ij}=1, 1\le i,j\le n\}$.  Conversely, given a digraph $D=(\mathcal{V},\mathcal{A})$ with a vertex set $\mathcal{V}=\{v_1,v_2,\ldots,v_n\}$ and an arc set $\mathcal{A}$, its {\it adjacency matrix} is defined as $A_D=(a_{ij})_{n\times n}$, where
\begin{equation*}
a_{ij}=\left\{\begin{array}{ll}
1,&\textrm{if } (v_i,v_j)\in \mathcal{A};\\
0,&\textrm{otherwise}.\end{array}\right.
\end{equation*}
 Given $A\in M_n\{0,1\}$, the $(i,j)$-entry of $A^k$ equals $t$ if and only if $D(A)$ has exactly $t$ distinct $ij$-walks of length $k$.

 Using adjacency matrices,  we have an equivalent definition for the stable index of digraphs as follows.

 {\bf Definition.} {\it   The stable index of a digraph $D$, denoted by $\theta(D)$, is defined to be the smallest integer $k$ such that $D$   contains at least two distinct $(k+1)$-walks with the same  initial vertex and  terminal vertex if such an integer exists; otherwise the stable index of $D$ is defined to be infinity. }

Given a digraph $D$,  by the above definitions we have $\theta(D)=\theta(A_D)$. Conversely, given a square 0-1 matrix $A$, we have $\theta(A)=\theta(D(A))$.

  Given a positive integer $n$, we denote  by $s(n)$ the maximum finite stable index of 0-1 matrices of order $n$.
 The precise value of $s(n)$ was determined in \cite{CHY}.  For $n\le 6$, we have $s(2)=1,s(3)=3,s(4)=4,s(5)=6,s(6)=7$; see \cite{CHY}. For $n\ge 7$ we have

\begin{theo} {\rm{\cite{CHY}}}\label{th1}
Let $n\geq7$ be an integer. Then
\begin{equation}\label{eq1}
s(n)=\begin{cases}
           \frac{n^2-1}{4},&\text{if $n$ is odd},\\
           \frac{n^2-4}{4},&\text{if  $n\equiv0$ (mod 4)},\\
           \frac{n^2-16}{4},&\text{if  $n\equiv2$ (mod 4).}
\end{cases}
\end{equation}
\end{theo}

 Similar with the  exponent of primitive matrices, we are interested in the following natural problem.

\begin{problem}
Given a positive integer $n$, which numbers can be the stable indices of  0-1 matrices (digraphs) of order $n$?
\end{problem}

We solve this problem in this paper. Denote by $\mathbb{Z}^+$  the set of positive integers and denote by $\text{LCM}(p,q)$  the least common multiple of two integers $p$ and $q$. For $k\in \mathbb{Z}^+$, we use $[k]$ to represent the  set $\{1,2,\ldots,k\}$.
 Let $\Theta(n)$ be the set of stable indices of 0-1 matrices of order $n$.
Then $\Theta(n)$ is also the set of indices of digraphs with $n$ vertices.  Our main result states as follows.

\begin{theo}\label{th2}
Let $n\ge 7$ be an integer. Then
\begin{equation*}\label{m1}
\Theta(n)=\left[s(n-1)+1\right]\cup\{{\rm{LCM}}(p,q):p+q=n, p,q\in \mathbb{Z}^+\}\cup\{\infty\}.
\end{equation*}
 \end{theo}
 For the sake of convenience, we prove Theorem \ref{th2} for digraphs.

\section{Proof of Theorem \ref{th2}}

Two digraphs $D_1=(\mathcal{V}_1,\mathcal{A}_1)$ and $D_2=(\mathcal{V}_2,\mathcal{A}_2)$ are {\it  isomorphic}, written $D_1\cong D_2$, if there is a bijection $\sigma: \mathcal{V}_1\rightarrow \mathcal{V}_2$ such that $(u,v)\in \mathcal{A}_1$ if and only if $(\sigma(u),\sigma(v))\in \mathcal{A}_2$. It is clear that two digraphs are isomorphic if and only if their adjacency matrices are permutation similar.  We say a digraph $D$ contains a {\it copy} of $H$ if $D$ has a subgraph isomorphic to $H$.

 In a digraph $D$, if there is a walk from $u$ to $v$ for all  $u,v\in\V(D)$, then $D$ is said to be {\it strongly connected}. A digraph  is strongly connected if and only if its adjacency matrix is irreducible.

 Denote by $\overrightarrow{C_k}$  a $k$-cycle.    Given an integer $k\ge 2$ and two disjoint    cycles $\overrightarrow{C_p}$ and  $\overrightarrow{C_q}$, let $\overrightarrow{g}(p,k,q)$ be the digraph  obtained by adding a $(k-1)$-path from a vertex of $\overrightarrow{C_p}$ to a vertex of $\overrightarrow{C_q}$, which has the following diagram.
 When $k=2$, we abbreviate $\overrightarrow{g}(p,k,q)$  as $\overrightarrow{g}(p,q)$.
 \begin{center}
   \begin{tikzpicture}[>=stealth]

\draw[->](0,0)arc[start angle =180,end angle=-180,radius=0.6cm];
\draw(0.6,0)node[scale=1]{$\overrightarrow{C_{p}}$};
\node[shape=circle,fill=black,scale=0.12](a)at (1.2,0){m};
\draw(1.4,-0.25)node[scale=1]{$v_{1}$};
\node[shape=circle,fill=black,scale=0.12](b) at (2.2,0){n};
\draw(2.2,-0.25)node[scale=1]{$v_{2}$};
\node[shape=circle,fill=white,scale=0.12](c)at (3.2,0){};
\node[shape=circle,fill=black,scale=0.12](d) at (3.45,0){};
\node[shape=circle,fill=black,scale=0.12](e) at (3.7,0){};
\node[shape=circle,fill=black,scale=0.12](f) at (3.95,0){};
\node[shape=circle,fill=white,scale=0.12](g)at (4.2,0){};

\node[shape=circle,fill=black,scale=0.12](h)at (5.2,0){r};
\draw(5,-0.25)node[scale=1]{$v_{k}$};

\draw[->](a)--(b);
\draw[->](b)--(c);
\draw[->](g)--(h);
\draw[->](6.4,0)arc [start angle =360,end angle=0,radius=0.6cm];
\draw(5.8,0)node[scale=1]{$\overrightarrow{C_{q}}$};
\draw(3.7,-1)node[scale=1]{$\overrightarrow{g}(p,k,q)$};
\end{tikzpicture}
 \end{center}
 It is clear that $$\theta(\overrightarrow{g}(p,k,q))=\text{LCM}(p,q)+k-2.$$
Denote by $C_p$ the adjacency matrix of the $p$-cycle $v_1\cdots v_pv_1$, which is $$\begin{bmatrix}& 1&&\\
                               &&\ddots&\\
                               &&&1\\
                               1&&&
                               \end{bmatrix}.$$

We need the following lemmas.
\begin{lemma}\label{le101}
Let $D$ be a strongly connected digraph of order $n\ge 2$. If $D$ is a cycle, then   $\theta(D)=
\infty$; otherwise, we have
$\theta(D)\le n$.
\end{lemma}
\begin{proof}
Recall that $\theta(D)=\theta(A_D)$. Since $D$ is strongly connected if and only if $A_D$ is irreducible, the result follows from \cite[Lemma 3]{CHY}.
\end{proof}
\begin{lemma}\label{le1}
 Let  $A=\left(\begin{array}{cc}C_p&X\\0&C_q\end{array}\right)\in M_{p+q}\{0,1\} $, where $X$ has exactly $k$ nonzero entries. Then
 $$\theta(A)\le  \big\lfloor  {\frac{pq}{k}}\big\rfloor.$$
\end{lemma}
\begin{proof}
Notice that   $$A^{m}=\left(\begin{matrix}{C_p}^{m}&
\sum_{k=0}^{m-1}{C_p}^kX{C_q}^{m-1-k}\\0&{C_q}^{m}\\
\end{matrix}\right)\equiv \left(\begin{matrix}{C_p}^{m} &B\\0& {C_q}^{m}\end{matrix}\right)$$ for all positive integer $m$. Since ${C_p}^kX{C_q}^{m-1-k}$ only changes the positions of $X$'s entries, each ${C_p}^kX{C_q}^{m-1-k}$ has exactly $k$ entries equal to 1. So the summation of all entries in $B$ is  $km$. If $m>\lfloor  pq/k\rfloor$, then $km>pq$, which implies that $B$ is not a 0-1 matrix. Therefore, $\theta(A)\le\lfloor  pq/k\rfloor$.
\end{proof}

\begin{corollary}\label{co1}
Let $p,q,n $ be positive integers such that $p+q=n$. Suppose $D$ is a digraph of order  $n$ containing a copy of $\overrightarrow{g}(p,q)$. If $\theta(D)>\max\{n,\lfloor  pq/2\rfloor\}$, then $D\cong\overrightarrow{g}(p,q)$.
\end{corollary}
\begin{proof}
 Since   $D$ contains a copy of $\overrightarrow{g}(p,q)$,  we may assume $$A_D=\left(\begin{matrix}{C_p}+A_{11} &
A_{12}\\A_{21}&{C_q}+A_{22} \\
\end{matrix}\right),$$
where $A_{12}\ne 0$.

If $A_{21}\ne 0$, then $A_D$ is irreducible. Applying Lemma \ref{le101} we have $\theta(D)\le n $, a contradiction. Therefore, we have $A_{21}=0$.

If $A_{11}\ne 0$, then since $C_p+A_{11}$ is irreducible, we have $\theta(D)\le \theta(C_p+A_{11}) \le p$, a contradiction. Hence, $A_{11}=0$. Similarly, we have $A_{22}=0$.

Now applying Lemma \ref{le1}, we can deduce that $A_{12}$ has exactly one nonzero entry. Therefore, $D\cong\overrightarrow{g}(p,q)$.
\end{proof}

\begin{lemma}\label{le2}
Suppose $p,q>0$ are relatively prime numbers. Let $kq\equiv x_k$ (mod $p$) for $1\le k\le p-1$. Then
$$\{x_1,x_2,\ldots,x_{p-1}\}=\{1,2,\ldots,p-1\}.$$
\end{lemma}
\begin{proof}
It suffices to prove $x_u\ne x_v$ for all $u>v$. Otherwise, suppose there exist $u,v\in\{1,2,\ldots,p-1\}$ such that $u>v$ and $x_u=x_v$. Then $x_u-x_v=(u-v)q$ is divided by $p$. Since $(p,q)=1$, we have  $p\mid u-v$,  which contradicts the fact $1\le u-v<p-1$. This completes the proof.
\end{proof}
\begin{corollary}\label{co2}
Suppose $p,q $ are relatively prime numbers such that $p>q>1$. Then
$$\min\{u: \text{ there exists a nonnegative number } v \text{ such that }p-q=uq-vp\}=p-1.$$
\end{corollary}
\begin{proof}
Notice that $(p-1)q\equiv p-q$ (mod $p$). Applying Lemma \ref{le2} we have the conclusion.
\end{proof}

Given two walks $\w_1$ and $\w_2$ in a digraph $D$, we denote by $\w_1\cup \w_2$ the union of $\w_1$ and $\w_2$, i.e.,  the subgraph of $D $ with vertex set $\V(\w_1)\cup \V(\w_2)$  and arc set $\A(\w_1)\cup\A(\w_2)$.
From the proof of \cite[Theorem 1]{CHY} we have the following lemma. For completeness we restate its proof.
\begin{lemma}\cite{CHY}\label{le9}
Suppose a digraph $D$   of order $n$ contains two distinct walks  $\w_1$ and $\w_2$ with the same length from $x$ to $y$ for some vertices $x$ and $y$.   If $\w_1\cup \w_2$  contains at most one cycle, then $\theta(D)<n-1$.
\end{lemma}
\begin{proof}
If $\w_1\cup \w_2$ is acyclic, then both $\w_1$ and $\w_2$ are directed paths with length less than $n-1$. Therefore, $\theta(D)<n-1$.

Now suppose $\w_1\cup \w_2$ contains exactly one cycle $\C_p$. If only one of the two walks, say $\w_1$, contains copies of $\C_p$, then $\w_2$ is a directed path with length less than $n$, which implies $\theta(D)<n-1$.
If both $\w_1$ and $\w_2$ contain copies of $\C_p$, then by deleting the same number of copies of $\C_p$ in both $\w_1$ and $\w_2$, we can get two distinct walks with the same length from  $x$ to $y$ such that  at least one of them contains no cycle. Again, we have $\theta(D)<n-1$.
\end{proof}

\par
For the sake of simplicity, we denote by
$$[a,b]=\{x\in \mathbb{Z}^+: a\le x\le b\}.$$ Now we are ready to present the proof of our main result.\\

 {\it Proof of Theorem \ref{th2}.}
 We first show
  \begin{equation}\label{eq4}
\Theta(n)\subseteq [s(n-1)+1]\cup\{\text{LCM}(p,q):p+q=n,p,q\in \mathbb{Z}^+\}\cup \{\infty\}.
\end{equation}
Suppose  $D$ has finite stable index $r$.  Then $D $ has two distinct $(r+1)$-walks $\w_1$ and $\w_2$   from $x$ to $y$ for some vertices $x,y$.

 If    $\w_1\cup \w_2$ contains at most one cycle,  then applying Lemma \ref{le9} we have $\theta(D)<n-1$. If $\w_1\cup \w_2$ contains two cycles whose vertex sets have nonempty intersection, then the subgraph $D_1$ induced by these two cycles is strongly connected. Applying Lemma \ref{le101}, we have $\theta(D)\le \theta(D_1)\le n$.
  Now suppose $\w_1\cup \w_2$ contains two disjoint cycles. We distinguish   three cases.

 {\it Case 1.} $\w_1\cup \w_2$ contains a copy of $\g(p,q)$ with $p+q=n$. Then applying Corollary \ref{co1}, we have either
 $\theta(D)\le \max\{n,\lfloor pq/2\rfloor\}$ or $D\cong\g(p,q)$, while the later case leads to $\theta(D)=\text{LCM}(p,q)$. Thus
 $$\theta(D)\in \left[  \max\{n,\lfloor n/2\rfloor\lceil n/2\rceil/2 \} \right]\cup\{\text{LCM}(p,q): p+q=n, p,q\in \mathbb{Z}^+\}.$$
Since $\max\{n,\lfloor n/2\rfloor\lceil n/2\rceil/2 \}\le s(n-1)+1$, we have
 $$\theta(D)\in \left[  s(n-1)+1 \right]\cup\{\text{LCM}(p,q): p+q=n, p,q\in \mathbb{Z}^+\}.$$

 {\it Case 2.} $\w_1\cup \w_2$ contains no copy of $\g(p, q)$ with $p+q=n$ and it contains a copy of $\g(p_1,k_1, q_1)$ for some positive integers $p_1,q_1,k_1$ with $p_1+q_1<n$.  Let
 $$\lambda_1(n)=\max\{\theta(\g(p,k,q)):p+q+k-2\le n, p+q\le n-1\}.$$

Note that
 $\theta(\g(p,k,q))=\text{LCM}(p,q)+k-2 $ and $$s(n)=\max\{\text{LCM}(p,q):p+q=n\}\quad \text{for}\quad n\ge 7.$$
 If $n=7$, since $s(6)=7$, we have $$\lambda_1(7)= \theta(\g(2,4,3))=8=s(6)+1.$$ For $n\ge 8$, since
$s(x)-x$ is strictly increasing on the integer variable  $x$ when $x\ge 2$, we have
$$\theta(\g(p,k,q))=\text{LCM}(p,q)+k-2\le s(p+q)+n-p-q\le s(n-1)+1.$$
Hence,
  $$\lambda_1(n)=s(n-1)+1 \quad \text{for }\quad n\ge 7$$
  and $$\theta(D)\le s(n-1)+1.$$

 {\it Case 3.} $\w_1\cup \w_2$  does not contain any copy of $\g(p,k,q)$.
 Then similarly as in the proof of Theorem 1 in \cite{CHY},  $\w_1\cup \w_2$ has the following diagram.
\begin{center}
 \begin{tikzpicture}[>=stealth]
\draw[->](1.5,0.5)arc[start angle =90,end angle=-270,radius=0.5cm];
\draw(1.5,0)node[scale=1]{$\overrightarrow{C_{p}}$};
\node[shape=circle,fill=black,scale=0.12](a)at (-0.5,-0.75){x};
\draw(-0.7,-0.8)node[scale=1]{$x$};
\draw[->](-0.5,-0.75)arc[start angle =160,end angle=85,radius=1.5cm];
\draw[->](-0.5,-0.75)arc[start angle =-160,end angle=-87,radius=1.5cm];
\draw(3.2,0.1)node[scale=1]{$\overrightarrow{w}_{1}$};
\draw[<-](3.5,-0.75)arc[start angle =20,end angle=95,radius=1.5cm];
\draw[<-](3.5,-0.75)arc[start angle =-20,end angle=-93,radius=1.5cm];
\draw(3.2,-1.7)node[scale=1]{$\overrightarrow{w}_{2}$};
\node[shape=circle,fill=black,scale=0.12](b) at (3.5,-0.75){y};
\draw(3.7,-0.8)node[scale=1]{$y$};
\draw[->](1.5,-2.25)arc [start angle =-90,end angle=-450,radius=0.5cm];
\draw(1.5,-1.75)node[scale=1]{$\overrightarrow{C_{q}}$};
\end{tikzpicture}
\end{center}
Suppose the $xy$-paths contained in $\w_1$  and   $\w_2$ have lengths $l$ and $t$, respectively. Then we have
\begin{eqnarray}
\theta(D)&=&\min\{l+up: l+up=t+vq \text{ for some nonnegative integers } u,v \}-1\nonumber\\
&\le& \min\{l+(q-1)p, t+(p-1)q\}-1.\label{eq2}
\end{eqnarray}
Next we show
\begin{equation*}\label{eq3}
\theta(D)\le s(n-1)+1.
\end{equation*}

If $p=q$, then we have $\theta(D)\le \max\{l,t\} < s(n-1)+1$. Next we assume $p> q$. Since $l\le n-1-q$, $t\le n-1-p$ and $p+q\le n-2$, by (\ref{eq2}) we have
\begin{eqnarray*}
\theta(D)&\le&t+(p-1)q-1\le n-2-p+(p-1)q=n-3+(p-1)(q-1)\equiv h(n,p,q).
\end{eqnarray*}
If $n$ is odd, then $\max\{h(n,p,q):p+q\le n-2,p>q\}$ is attained at $(p,q)=((n-1)/2,  (n-3)/2).$ Hence, by (\ref{eq1}) we have
\begin{eqnarray*}
\theta(D) \le  h(n,p,q)
 \le  n-3+(n-3)(n-5)/4
 = (n^2-4n+3)/4
\le  s(n-1)+1.
\end{eqnarray*}
If $n$ is even, then $\max\{h(n,p,q):p+q\le n-2,p>q\}$ is attained at $(p,q)=(n/2,  n/2-2).$
Hence, by (\ref{eq1}) we have
\begin{eqnarray*}
\theta(D) \le  h(n,p,q)
 \le  n-3+(n/2-1)(n/2-3)
 = (n^2-4n)/4
\le s(n-1)+1.
\end{eqnarray*}

Combining all the above  cases we have (\ref{eq4}).
Now we prove
\begin{equation}\label{eq5}
[s(n-1)+1]\cup\{\text{LCM}(p,q):p+q=n,p,q\in \mathbb{Z}^+\}\cup \{\infty\}\subseteq \Theta(n).
\end{equation}
Notice that $\theta(\overrightarrow{g}(p,q))=\text{LCM}(p,q)$, $\theta(\overrightarrow{C}_n)=\infty$ and $\theta(\overrightarrow{K}_n)=1$, where $\overrightarrow{K}_n$ is the complete digraph of order $n$. It suffices to verify  \begin{equation}\label{eq6}
[2,s(n-1)+1]\subseteq \Theta(n).
\end{equation}

Let $G(p,q)$ be the set of graphs with the following diagrams, where $\overrightarrow{C}_p$ is a $p$-cycle $u_1u_2\cdots u_{p}u_1$, $\overrightarrow{C}_q$ is a $q$-cycle $v_1v_2\cdots v_{q}v_1$, the arrows from $x$ to $\overrightarrow{C}_p$ and $\overrightarrow{C}_q$ represent  two paths $xx_1\cdots x_r u_1$ and $xy_1\cdots y_s v_1$, the arrow from $\overrightarrow{C}_p$ to $\overrightarrow{C}_q$ represents an arc $u_iv_j$ with $1\le i\le p$, $1\le j\le q$.
\begin{figure}[H]
		\centering
		\includegraphics[width=1.5in]{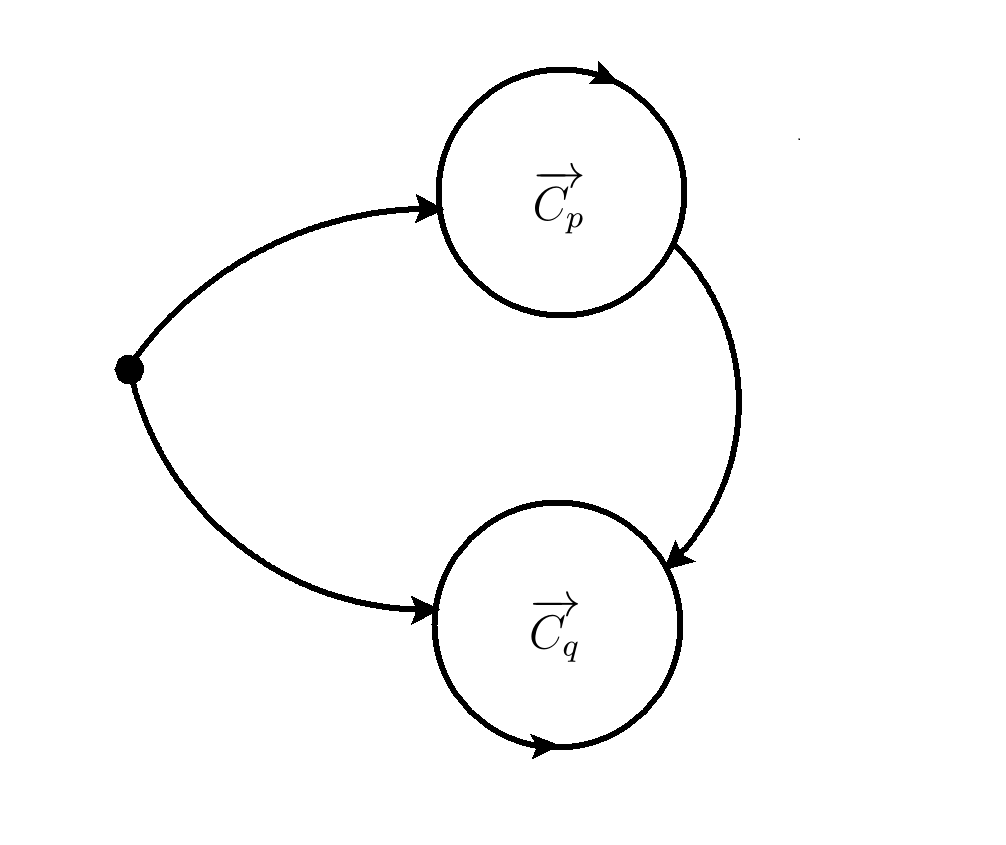}\hspace{0.8cm}
      \vskip -2.2 cm \hspace{-4cm} $x$
      \vskip 1.7cm
      $G(p,q)$
\end{figure}
\vskip -0.5cm
\noindent Let $G(p,q,l,t)$ be the  set of graphs in $G(p,q)$ such that the length of the path  $xx_1\cdots x_r u_1\cdots u_iv_j$ is $l $ and  the length of the path  $xy_1\cdots y_s v_1\cdots  v_j$ is $t$. Then
$$\theta(D)=\min\{l+ap: l+ap=t+bq \text{ for some nonnegative integers }a, b\}-1 $$
for all $D\in G(p,q,l,t)$. We denote by $\theta(G(p,q,l,t))$ the stable index of digraphs in $G(p,q,l,t)$.

If $q=p-1$ and $l-t=d<p$, applying Lemma \ref{le2},
\begin{eqnarray*}
&&\min \{l+ap: l+ap=t+bq \text{ for some nonnegative integers }a, b\}\\
&&=\min \{l+ap: l-t= bq-ap \text{ for some nonnegative integers }a, b\}
\end{eqnarray*}
is attained at $(a,b)=(p-d-1,p-d)$.
Therefore,
$$\theta(G(p,p-1,l,t))=t-1+(p-1)(p-d).$$

Denote by $f(p,q,d)=\{\theta(G(p,q,l,t)):l-t=d\}$.  If $q=p-1\ge 1$ and $d=1$,  then we have $1\le t\le \lceil n/2\rceil-1$ in $G(p,q,l,t)$. Hence,
\begin{eqnarray*}
 f(p,p-1,1)&=&\{w+(p-1)^2:  0\le w\le \lceil n/2\rceil-2\}\\ &=&[(p-1)^2,(p-1)^2+\lceil n/2\rceil-2].
\end{eqnarray*}
If $q=p-1\ge 2$ and $d=2$,  then we have $1\le t\le \lfloor n/2\rfloor-1$ in $G(p,q,l,t)$. Hence,
\begin{eqnarray*}
  f(p,p-1,2)&=&\{w+(p-1)(p-2):  0\le w\le \lfloor n/2\rfloor-2\}\\
  &=&[(p-1)(p-2),(p-1)(p-2)+\lfloor n/2\rfloor-2].
\end{eqnarray*}
Notice that
\begin{eqnarray*}
(p-1)^2\le (p-1)(p-2)+\lfloor n/2\rfloor-1 \quad\text{for}\quad 2\le p\le \lfloor n/2\rfloor, \\
(p-1)(p-2)\le  (p-2)^2+\lceil n/2\rceil-1 \quad\text{for}\quad  2\le p\le \lceil n/2\rceil.
\end{eqnarray*}
We have
\begin{eqnarray*}
\bigcup_{p=3}^{\lfloor n/2\rfloor}\big\{f(p,p-1,1)\cup f(p,p-1,2)\big\} = \big[2,(\lfloor n/2\rfloor-1)^2+\lceil n/2\rceil-2\big]\equiv T(n),
\end{eqnarray*}
which implies  $T(n)\subseteq \Theta(n)$.

Next we distinguish three cases.

{\it Case 1.}   $n$ is even, say, $n=2m$. Then $$(\lfloor n/2\rfloor-1)^2+\lceil n/2\rceil-2=m^2-m-1$$
and $$s(n-1)+1=(n^2-2n+4)/4=m^2-m+1.$$
Note that $$\theta(\overrightarrow{g}(m,m-1))=m^2-m\quad \text{and}\quad  \theta(\overrightarrow{g}(m,3,m-1))=m^2-m+1.$$
We have (\ref{eq6}).

{\it Case 2.} $n \equiv 1$ (mod 4), say, $n=4k+1$. Suppose $m=2k$. Then  $$(\lfloor n/2\rfloor-1)^2+\lceil n/2\rceil-2=m^2-m $$
and $$s(n-1)+1=(n^2-2n+1)/4=m^2.$$
Since $m+1$ and $m-1$ are relatively prime, applying Corollary \ref{co2}, we have
\begin{eqnarray*}
f(m+1,m-1,2)=\{w+m(m-1): 0\le w\le m-2\}=[m^2-m,m^2-2].
\end{eqnarray*}
Moreover, since $m$ is even, we have
$$\theta(\overrightarrow{g}(m+1,m-1))=m^2-1  \quad \text{and}\quad \theta(\overrightarrow{g}(m+1,3,m-1))=m^2=s(n-1)+1.$$
Therefore, we have (\ref{eq6}).

{\it Case 3.} $n \equiv 3$ (mod 4), say, $n=4k+3$. Suppose $m=2k+1$. Then  $$(\lfloor n/2\rfloor-1)^2+\lceil n/2\rceil-2=m^2-m $$
and $$s(n-1)+1=(n^2-2n-11)/4=m^2-3.$$
Since $m+2$ and $m-2$ are relatively prime, applying Corollary \ref{co2} we have
$$f(m+2,m-2,4)=\{w+(m+1)(m-2):0\le w\le m-3\}=[m^2-m-2,m^2-5].$$
On the other hand,
$$\theta(\overrightarrow{g}(m+2,m-2))=m^2-4  \quad \text{and}\quad \theta(\overrightarrow{g}(m+2,3,m-2))=m^2-3=s(n-1)+1.$$
Therefore, we get (\ref{eq6}).

\par
Combining (\ref{eq4}) and (\ref{eq5}), we obtain
$$\Theta(n)=\left[s(n-1)+1\right]\cup\{\text{LCM}(p,q):p+q=n,p,q\in \mathbb{Z}^+\}\cup\{\infty\}.$$
This completes the proof.   \hspace{11cm} $\square$

 {\bf Remark.} Denote by $\overrightarrow{L}_n$ the union of the $k$-cycle $v_1v_2\cdots v_kv_1$ and the 2-cycle $v_1v_nv_1$, where $k=n$ if $n$ is odd and $k=n-1$ if $n$ is even. Then we have $$\theta( \overrightarrow{L}_n)=n \quad \text{for}\quad  n\ge 3 .$$ Note that $$\theta(\overrightarrow{K}_n)=1, \theta(\overrightarrow{g}(1,2))=2, \theta(\overrightarrow{g}(2,3))=6,  \theta(\overrightarrow{g}(2,3,3))=7.$$
 We have  $$\Theta(n)=[s(n)]\cup \{\infty\}  \quad \text{for}\quad  1\le n\le 6 .$$
   When $n\ge 7$, by direct computation  we see that $\Theta(10)=[s(10)]\cup \{\infty\}$ and  $$\left[s(n-1)+1\right]\cup\{\text{LCM}(p,q):p+q=n, p,q\in \mathbb{Z}^+\}\ne [s(n)] \quad\text{for}\quad  n\ne 10.$$ So $\Theta(n)$ has gaps in the set $[s(n)]$ for $n\in \mathbb{Z}^+-[6]\cup\{10\}$.
\section*{Acknowledgement}

The authors are grateful to Professor Xingzhi Zhan for valuable suggestions. This work was supported by the National Natural Science Foundation of China (No. 12171323),  the Science and  Technology Foundation of Shenzhen City (No. JCYJ20190808174211224, No. JCYJ20210324095813036),  Guangdong Basic and Applied Basic Research
Foundation (No. 2022A1515011995) and a project of Educational Commission of Guangdong Province of China (2021KTSCX103).

\end{document}